\documentclass[preprint,12pt,authoryear]{elsarticle}

\usepackage{graphicx}
\usepackage{hhline}
\usepackage{amsmath,amsthm,amsfonts,amscd,amssymb,mathrsfs}
\usepackage{xspace}
\usepackage{hyperref}
\usepackage{booktabs} 
\usepackage{nicematrix}
\usepackage{physics}

\usepackage{array,multirow}   
\newcolumntype{C}{>{$}c<{$}} 
\newcolumntype{L}{>{$}l<{$}} 
\newcolumntype{R}{>{$}r<{$}} 

\usepackage{verbatim}
\usepackage{amscd}   
\usepackage[all]{xy} 

\usepackage{nicefrac}
\usepackage{xfrac}

\usepackage{mathdots}
\usepackage{tikz}

\usepackage[T1]{fontenc} 
\usepackage{cleveref} 
\numberwithin{equation}{section} 
\usepackage{color}

\newcommand{\CC}{\mathbb{C}}
\newcommand{\RR}{\mathbb{R}}

\newcommand{\QQ}{\mathbb{Q}}

\newcommand{\ccirc}[1]{\xymatrix@1{ +<1ex>[o][F-]{#1}}}

\usepackage{thmtools}
\declaretheorem[style=plain, numberwithin=section]{theorem}
\declaretheorem[style=definition,name=Definition,qed=$\blacksquare$, numberwithin=section, sibling=theorem]{definition}
\declaretheorem[style=definition,name=Example,qed={$ \diamondsuit$}, numberwithin=section, sibling=theorem]{example}
\declaretheorem[style=definition,name=Remark,qed=$\blacksquare$, numberwithin=section, sibling=theorem]{remark}

\journal{Mechanism and Machine Theory }

\begin{document}

\begin{frontmatter}

\title{Equilibria for Networks of Linear Translational Springs}

\author[AU]{Luke Oeding}\ead{oeding@auburn.edu}
\author[AU]{Ethan Clayton}\ead{ewc0025@auburn.edu}
\author[AU]{Jackson Elsea}\ead{jwe0023@auburn.edu}
\affiliation[AU]{organization={Auburn University},
            addressline={Department of Mathematics and Statistics},
            city={Auburn},
            state={AL},
            country={USA}
            }
\author[NYU]{Nicholas Wang}\ead{nyw2006@nyu.edu}
\affiliation[NYU]{organization = {New York University},
                  addressline = {Courant Institute of Mathematical Sciences},
                  city={New York},
                  state = {NY},
                country = {USA}}

\author[AFRL]{Adam Rutkowski}\ead{adam.rutkowski@us.af.mil}
\affiliation[AFRL]{organization = {Air Force Research Laboratory},
city = {Eglin Air Force Base},
state = {Florida}, 
Country = {USA}
}

\begin{abstract}
We use tools from nonlinear algebra to study the equilibria of small linear translational spring networks. Specifically we use the techniques of homotopy continuation, monodromy, and parameter homotopy (a.k.a. cheater homotopy) to solve all rigid linear translational spring networks up to $5$ nodes in both $2$ and $3$ dimensions. We describe a method of implementing parameter homotopy that arises naturally from the physical structure of the system. We give precise total degree bounds on the maximum number of solutions for general planar spring networks. We discuss further efficiency gains obtained from polyhedral homotopy methods. We compare the computation efficiency of these techniques against a baseline of Newton's method. 
\end{abstract}


\end{frontmatter}

\section{Introduction}
A linear translational spring network is a graph (typically embedded in $\mathbb{R}^2$, or $\mathbb{R}^3$) whose vertices have some position in space, and whose edges are linear translational springs with a known spring constant and resting length. We use the term "translational" to refer to a spring that exerts a force that depends on the length of the spring compared to its resting length, and "linear" to indicate that the force is proportional to the change in length (i.e. that it obeys Hooke's law). Other types of springs, for instance torsional (i.e. rotational), could be considered, and these springs could be linear in the sense that they exert a torque that is proportional to the an angular displacement, but for this article we only consider translational springs. Spring networks can be useful models in a variety of scientific domains. 

The equilibria, that is configurations of nodes and springs which yield no net force, of certain spring networks have been studied before. \cite{pigoski1995inverse} used an inverse force analysis to study the statics of a spring system with $2$ springs and $3$ nodes. This method uses algebraic elimination techniques to reduce the problem to a single degree $6$ polynomial whose roots give the equilibria of the network. A similar inverse force analysis was used to solve a $3$ spring $4$ node spring network in \cite{zhang1997reverse,ZhangLiao}, which gave a degree $22$ polynomial describing the equilibria positions. 

Solving a system via elimination theory is the polynomial analogue to Gaussian elimination. The original system is transformed into a triangular system of equations that can be solved via back substitution. In the polynomial case, the last polynomial depends on only one variable, and can be solved numerically, for instance via eigenvalue methods. For each root of the last polynomial one substitutes that root to obtain a new triangular system of equations with one fewer variable and one fewer equation. Then one repeats this process until all roots are found.

Using elimination theory has the advantage of giving solutions in terms of the roots of a polynomial whose coefficients are given by algebraic expressions in terms of known values. Therefore, the roots can be quickly computed to high degrees of precision without needing significant computation power. Furthermore it gives an upper bound on the number of equilibria of the network, namely the degree of the polynomial. However, when applying this method to larger spring networks the degree of the polynomial that describes the equilibria grows large. This makes computing solutions more difficult. Additionally the elimination techniques used in these methods involve dividing by expressions. This can cause the method to miss `singular' solutions where these expressions are $0$. 

We utilize methods from numerical algebraic geometry to address these issues and to allow for larger and more complex spring networks to be solved. The primary technique that we use is homotopy continuation. Homotopy continuation has been used to solve systems of polynomials arising in a variety of physical domains. For some examples in robotics see \cite{tchon2007robotics}, \cite{moradi2012robotics}, and \cite{luo2022forwardstatics}. Additionally it has been used to solve the inverse position problem for a $6$ arm linkage in \cite{Wampler6R}. 

Homotopy continuation has also been used to solve systems arising from spring networks. For example \cite{su2006planarcompliant,su2028planarcompliant} use homotopy continuation to solve for the statics of a planar compliant mechanism consisting of $3$ translational and $3$ torsional springs. 

There are several programs that perform homotopy continuation including Bertini, \cite{Bertini}, the \texttt{NumericalAlgebraicGeometry} package in Macaulay2, \cite{leykin2011numerical,M2}, and the Julia library \texttt{HomotopyContinuation.jl}, \cite{hc.jl}, which we used primarily for this work. For brevity, we refer to this library as \texttt{hc.jl}.

We create a system of polynomial equations whose solutions are the equilibria of the network. We review our various methods for solving these systems in \Cref{sec:methods}. We demonstrate the effectiveness of several methods: homotopy continuation (\Cref{sec:HC}), monodromy (\Cref{sec:monodromy}), and parameter homotopy (\Cref{sec:ph}). We show that these methods can be significantly more computationally efficient than a classical application of Newton's method to the polynomial system (see \Cref{sec:newton}). While certainly there are more sophisticated solvers, we chose to compare to Newton's method in order to have a common baseline for comparison and to highlight the difference between a solver that finds one solution at a time (Newton-based methods) and solvers that attempt to find all solutions at once (homotopy methods). 

We identify the Bézout bounds for the number of solutions on two different formulations of the system of equations for spring networks (see \Cref{thm:Bézout}) and find the optimal choices of base points in \Cref{thm:bound}. In \Cref{sec:polyhedral} we discuss the advantages of utilizing polyhedral start systems for homotopy continuation solving. In particular we find that the bounds for the number of solutions given by polyhedral methods can be several orders of magnitude smaller than the Bézout bound on the number of solutions for the same system. In \Cref{sec:2d} we report on our computations for all small 2- and 3-dimensional rigid linear translational spring networks with up to 5 nodes. In \Cref{sec:big} we discuss strategies for dealing with larger networks.

\section{Methods}\label{sec:methods}
Our main strategy is to transform the problem of finding the equilibria of a spring network into finding the zeros of a system of polynomial equations. We give two distinct constructions for a system of polynomials in \Cref{sec:Standard Formulation} and \Cref{sec:Inverse Formulation}. We then outline a range of different methods for solving these systems of equations and evaluate their effectiveness on spring networks of different sizes.

\subsection{From Spring Networks to Systems of Equations}\label{sec:Standard Formulation}
In this section we discuss the systems of equations that govern linear spring network equilibria.

\begin{definition}
    A translational spring network is a pair $N = (P, S)$ with set of nodes (points) $P = \{ p_1, p_2, \ldots, p_n \}$ and a collection of linear springs $S \subseteq \{ s_{i,j}: 1 \leq i, j \leq n, i < j \}$ (with $s_{i, j}$ connecting points $p_i$ and $p_j$). Each spring $s_{i,j} \in S$ is given by a pair $s_{i,j} = (k_{i,j}, \tilde{\ell}_{i,j})$ where $k_{i,j}$ is the spring constant and $\tilde{\ell}_{i,j}$ is the resting length of $s_{i,j}$.

    An embedding of a network $N=(P,S)$ into $\RR^d$ is a function $f: P \to \mathbb{R}^d$ that assigns coordinates to the points and, in turn, determines the stretched lengths $\ell_{i,j}=|f(p_i)-f(p_j)|$.  
    We call such an assignment an \emph{equilibrium position} if for each $1 \leq i \leq n$ the net force on each point $p_i \in P$ from all of the springs incident to $p_i$ is $0$.
\end{definition}

Our goal is to compute all of the equilibrium positions for a given spring network. A spring network is called rigid if it has finitely many equilibrium positions up to rotations and translations. Given a spring network $N = (P, S)$, we construct a system of polynomial equations in multiple variables, whose common roots are the equilibrium positions of $N$. For simplicity we will first make the assumption that $N$ is a spring network in dimension $2$ and then demonstrate how our methods can be extended to higher dimensions.  

The position of each node $p_i$ is unknown. Let $p_j = (x_j, y_j)$ for variables $x_1, \ldots, x_n$ and $y_1, \ldots, y_n$. Similarly the actual lengths of each spring $s_{i,j}$ is an unknown quantity. Let $\ell_{i,j}$ be a variable representing the actual length of $s_{i, j}$, $\tilde{\ell}_{i,j}$ be a variable representing the resting length of $s_{i,j}$, and $k_{i,j}$ be a variable representing the spring constant of $s_{i,j}$. 

Thus, the set of variables is $\{x_1, \ldots, x_n, y_1, \ldots, y_n \} \cup \{ \ell_{i, j} | s_{i, j} \in S \}$. The set of all spring constants and resting lengths are known values and so $\{ k_{i,j} | s_{i,j} \in S \} \cup \{ \tilde{\ell}_{i,j} | s_{i,j} \in S \}$ is the set of parameters. 

We now construct a system of equations in terms of the variables and parameters from the physical constraints of the spring network. Thus, given some fixed parameter values we can solve the system of equations in terms of the variables and obtain the equilibria of the spring network. 

The first type of equations are geometric constraints. If a spring $s_{i,j}$ exists between points $p_i$ and $p_j$ then the length of the spring must be equal to the distance between $p_i$ and $p_j$. Algebraically this gives an equation of the form 
\begin{equation}\label{eq:standardgeom}
    (x_j - x_i)^2 + (y_j - y_i)^2 = (\ell_{i,j})^2,
\end{equation}
where we have squared the length to make the condition a polynomial, which also allows for the possibility of negative spring lengths.

The other type of equation comes from balancing forces at nodes. Using Hooke's Law we know that the force that spring $s_{i,j}$ exerts on node $p_i$ in the direction of node $p_j$ is
\[
f_{i,j} = \begin{pmatrix} (f_{i,j})_x \\(f_{i,j})_y \end{pmatrix} = 
k_{i,j} (\ell_{i,j} - \tilde{\ell}_{i,j}) \frac{1}{(\ell_{i,j})}\begin{pmatrix}   x_j - x_i \\ y_j - y_i\end{pmatrix},
\]
with the $x$- and $y$-components of the force denoted $(f_{i,j})_x $ and $(f_{i,j})_y $.
 The net force on node $p_i$ is the sum of the forces generated by each spring attached to nodes adjacent to $p_i$. Thus, the $x$- and $y$-components of the force are
\begin{equation}
\begin{matrix}
    (F_i)_x = \sum_{s_{i,j} \in S} (f_{i,j})_x = \sum_{s_{i,j} \in S}  k_{i,j} (\ell_{i,j} - \tilde{\ell}_{i,j})   \frac{(x_j - x_i)}{(\ell_{i,j})},
\\
    (F_i)_y  = \sum_{s_{i,j} \in S} (f_{i,j})_y = \sum_{s_{i,j} \in S}  k_{i,j} (\ell_{i,j} - \tilde{\ell}_{i,j})   \frac{(y_j - y_i)}{(\ell_{i,j})}.
    \end{matrix}
\end{equation}

These equations are not polynomial in our variables. However, we can rectify this by simple algebraic manipulation. Since the spring network is in equilibria, the net force on each node, in the $x$ and $y$ direction must be $0$. Thus, 
\[\sum_{s_{i,j} \in S}  k_{i,j} (\ell_{i,j} - \tilde{\ell}_{i,j}) \frac{(x_j - x_i)}{(\ell_{i,j})} = 0\quad \text{and} \quad \sum_{s_{i,j} \in S}  k_{i,j} (\ell_{i,j} - \tilde{\ell}_{i,j})   \frac{(y_j - y_i)}{(\ell_{i,j})}=0.
\]
We can now multiply both equations by $\prod_{s_{i,j} \in S} \ell_{i,j}$ to remove the denominators. This gives the following equations for each node $p_i$ in the spring network:
\begin{equation}
\begin{matrix}\label{eq:standardforce}\displaystyle
    \sum_{s_{i,j} \in S}  k_{i,j} (\ell_{i,j} - \tilde{\ell}_{i,j}) (x_j - x_i) \prod_{s_{i,k} \in S, j \neq k} \ell_{i,k} = 0,
\\ \displaystyle
\sum_{s_{i,j} \in S}  k_{i,j} (\ell_{i,j} - \tilde{\ell}_{i,j}) (y_j - y_i) \prod_{s_{i,k} \in S, j \neq k} \ell_{i,k} = 0.
\end{matrix}
\end{equation}
These equations are polynomial in our desired variables.

We fix a coordinate frame by specifying coordinates of nodes.
In the case where $d = 2$ we can do this by asserting that $p_1 = (0, 0)$ and that $p_2 = (x_2, 0)$ for some $x_2 \in \mathbb{R}$. This allows us to ignore the net force equations for $p_1$. Similarly since $p_2$ is fixed on the $x$-axis at some position $(x_2, 0)$ we may ignore the net force equation in the $y$-direction for $p_2$. As a result we can remove  $x_1, y_1$ and $y_2$ from our set of variables and delete any equations that contain those variables.

\subsection{Using Inverse Lengths}
\label{sec:Inverse Formulation}

An alternate method of constructing a system of equations that represents the spring network (as long as the springs have non-zero length) is using inverse spring lengths. Let $I_{i,j} = \frac{1}{\ell_{i,j}}$ and construct a system with variables $\{x_2, \ldots, x_n, y_3, \ldots, y_n \} \cup \{ I_{i, j} | s_{i, j} \in S \}$ and parameters $\{ k_{i,j} | s_{i,j} \in S \} \cup \{ \tilde{\ell}_{i,j} | s_{i,j} \in S \}$.

For any spring $s_{i,j}$ between points $p_i$ and $p_j$ we can rewrite the equations that arise from the geometric constraints as
\begin{equation}\label{eq:invgeom}
    ((x_j - x_i)^2 + (y_j - y_i)^2) (I_{i,j})^2 = 1.
\end{equation}

We can also rewrite the force balancing equations in terms of our new variables. By Hooke's Law the magnitude of force that spring $s_{i,j}$ exerts on node $i$ is given by $F = k_{i,j} (\ell_{i,j} - \tilde{\ell}_{i,j})$. Using the relation $I = \frac{1}{\ell}$ we have the vector valued force equation
\begin{equation}\label{eq:invforce}
    F = \sum_{s_{i,j} \in S} k_{i,j} (\ell_{i,j} - \tilde{\ell}_{i,j})   \frac{1}{\ell_{i,j}} \begin{pmatrix} x_j - x_i \\ y_j - y_i\end{pmatrix} 
    = \sum_{s_{i,j} \in S} k_{i,j} (1 - \tilde{\ell}_{i,j} I_{i,j})\begin{pmatrix} x_j - x_i \\ y_j - y_i\end{pmatrix}.
\end{equation}

This equation is already polynomial in the variables and parameters we chose. The net force equations for each node in the $x$- and $y$-directions are given by sums of these equations and are therefore also polynomial. Thus, we have a system of polynomial equations whose solution set defines the equilibria of the spring network. 

Similar to the previous setup, any equilibrium remains an equilibrium under translations and rotations in space. Therefore, if we assert that $p_1 = (0, 0)$ and that $p_2 = (x_2, 0)$ for some $x_2 \in \mathbb{R}$, we may again ignore the net force equations on $p_1$ and the net force equation in the $y$-direction for $p_2$. Additionally we can remove $x_1, y_1$ and $y_2$ from our set of variables.

Using inverse spring lengths as variables instead of spring lengths reduces the degrees of the polynomials in the system of equations. This means that it is faster to compute solutions with both homotopy continuation and monodromy which we show theoretically in \Cref{sec:Comparing Constructions} and experimentally in \Cref{sec:HC}. However, if any solution has one of the spring lengths set to $0$ then this solution will not be a solution of the system formulated with inverse spring lengths since the corresponding inverse-length will be undefined. As a result this method removes any singular solution where the length of a spring is $0$. In both of these constructions we have reduced the problem to finding the solutions to a set of polynomial equations. We can now use a variety of existing numerical techniques for solving polynomials systems in order to find the solutions.

\begin{example}[A $2$-dimensional $3$-node spring network]
Consider a spring network $N = (P, S)$ with $P = (p_1, p_2, p_3)$ (with $p_j = (x_j, y_j)$) and $S = \{ s_{1,2}, s_{1,3}, s_{1,2} \}$, depicted at \Cref{fig:K3network}. 
The formulation of the conditions for equilibria in terms of lengths produces the following system of equations
\renewcommand{\arraystretch}{1.25}
\begin{equation}\label{eq:3node1}
\begin{array}{l}
  (x_2)^2-(\ell_{1,2})^2 = 0, \\
  (x_3)^2 + (y_3)^2- (\ell_{1,3})^2 = 0,\\
  (x_3-x_2)^2 + (y_3)^2 - (\ell_{1,3})^2 = 0,\\
  x_2 k_{1,2} (\ell_{1,2} - \tilde{\ell}_{1,2}) \ell_{2,3} + (x_2 - x_3) k_{2,3} (\ell_{2,3} - \tilde{\ell}_{2,3}) \ell_{1,2} = 0, \\
  x_3 k_{1,3} (\ell_{1,3} - \tilde{\ell}_{1,3}) \ell_{2,3} + (x_3 - x_2) k_{2,3} (\ell_{2,3} - \tilde{\ell}_{2,3}) \ell_{1,3} = 0,  \\
  y_3 k_{1,3} (\ell_{1,3} - \tilde{\ell}_{1,3}) \ell_{2,3} + y_3 k_{2,3} (\ell_{2,3} - \tilde{\ell}_{2,3}) \ell_{1,3} = 0.  
  \end{array}
\end{equation} 
On the other hand, the formulation in terms of inverse-lengths gives the following system of equations
\begin{equation}\label{eq:3nodeInv}
\begin{array}{l}
  (x_2)^2 (I_{1,2})^2 - 1= 0, \\
  ((x_3)^2 + (y_3)^2) (I_{1,3})^2 - 1 = 0, \\
  ((x_3-x_2)^2 + (y_3)^2) (I_{1,3})^2 - 1 = 0, \\
  x_2 k_{1,2} (1 - I_{1,2} \tilde{\ell}_{1,2}) + (x_2 - x_3) k_{2,3}(1 - I_{2,3} \tilde{\ell}_{2,3}) = 0,  \\
  x_3 k_{1,3} (1 - I_{1,3} \tilde{\ell}_{1,3}) + (x_3 - x_2) k_{2,3} (1 - I_{2,3} \tilde{\ell}_{2,3}) = 0,  \\
  y_3 k_{1,3} (1-I_{1,3} \tilde{\ell}_{1,3}) + y_3 k_{2,3} (1 -I_{2,3} \tilde{\ell}_{2,3}) = 0. 
   \qedhere
  \end{array}
\end{equation} 
\end{example}

\subsection{Comparing Constructions}
\label{sec:Comparing Constructions}
A reference for this section is \cite[Chapter~8.7]{CLO_text}.
Given a system of $n$ homogeneous polynomial equations in $n+1$ variables with complex coefficients, Bézout's Theorem implies that the number of common solutions in projective space, counted with multiplicity, is either infinite (with probability zero if the equations are general) or equal to the product of the degrees of the polynomials. Thus,  in the finite dimensional case the number of distinct solutions is bounded above by the product of the degrees. 

We can use this to obtain a bound on the number of solutions to our systems. For both the standard and inverse-length constructions, the number of polynomials and number of variables are equal. However, the polynomials are inhomogeneous. Thus, the system of equations is given $n$ inhomogeneous polynomials in $n$ variables. We can imagine that each is a polynomial in $n + 1$ variables and homogenize each polynomial with respect to the $n + 1$'st variable. Homogenizing does not change the degree of any of the polynomials, so we obtain a system of $n$ homogeneous polynomial equations in $n+1$ variables. Thus, the Bézout bound for the number of solutions is the product of the degrees of the polynomials. 

The best known algorithms for computing the common roots of the system of polynomials, given in \cite{Faugere2013PSSFLA}, are polynomial in the Bézout bound. Thus, we can asymptotically bound the runtime for computing equilibria for both the standard system described in \ref{sec:Standard Formulation} and the inverse-lengths system described in \ref{sec:Inverse Formulation}. We compute the Bézout bound for each in the case of dimension $2$.

\begin{theorem}\label{thm:Bézout}
Let $N = (P, S)$ be a rigid spring system in dimension $2$. Assume that $N$ has $n$ nodes and $s$ linear springs. 
The Bézout bounds $B_f$ of the standard system and $B_i$ of the inverse-length system are respectively given by 
\begin{equation}\label{eq:bf_bound}
    B_f =  2^s \prod_{j = 1}^n \left(1 + d(p_j)\right)^2 \leq 2^s \left(1 + \frac{2s}{n}\right)^{2n},
\end{equation}
and
\begin{equation}\label{eq:bi_bound}
    B_i = 4^s 4^n = 2^{2s+2n}.
\end{equation}
\end{theorem}
\begin{proof}
In the formulation of the standard system there are two types of polynomials: ones that derive from geometric constraints and ones that derive from force balancing equations. For each spring $s_{i,j} \in S$ we obtain an equation given by geometric constraints derived in \Cref{eq:standardgeom}. This equation has degree $2$. Similarly for every node $p_i \in P$, we obtain two force balancing equations derived in \Cref{eq:standardforce}. Each of these equations has degree equal to $1 + d(p_i)$, where $d(p_i)$ is the number of springs adjacent to node $i$. Therefore, the Bézout bound $B_f$ of the standard system is given by 
\[
B_f = \prod_{i = 1}^{s} 2 *\prod_{j = 1}^n (1 + d(p_j))^2 = 2^s \left[\prod_{j = 1}^n (1 + d(p_j))\right]^2.
\]

Since $N = (P, S)$ is rigid we know that $d(p_j) \geq 2$. Thus
\[
    B_f \geq 2^s \prod_{j = 1}^n (3)^2 = 2^s9^n.
\]
Using the AM-GM inequality we can make the following bound:
\[
\prod_{j = 1}^n (1 + d(p_j)) \leq \left( \frac{\sum_{j = 1}^n (1 + d(p_j))}{n} \right)^n = \left(1 + \frac{2s}{n}\right)^n.  
\]
Therefore
\[
    B_f = 2^s \prod_{j = 1}^n (1 + d(p_j))^2 \leq 2^s \left(1 + \frac{2s}{n}\right)^{2n}.
\]

In the case where the underlying graph is a complete graph $d(p_j) = n - 1$ for all $j$ and 
\begin{equation}   B_f = 2^s \prod_{j = 1}^n (1 + (n - 1))^2 = 2^s \prod_{j = 1}^n n^2 = 2^sn^{2n}.
\end{equation}

In the formulation of the inverse-length system we again have polynomials derived from both geometric constraints and force balancing equations. For each spring $s_{i,j} \in S$ the equation given by geometric constraints is derived in \Cref{eq:invgeom} and has degree $4$. Similarly for every node $p_i \in P$, we derive two force balancing equations in \Cref{eq:invforce} which both have degree $2$. Therefore, the Bézout bound $B_i$ of the inverse-length system is given by
\[
    B_i = \prod_{i = 1}^{s} 4 *\prod_{j = 1}^n 2^2 = 4^s 4^n = 2^{2s+2n}.\qedhere
\]
\end{proof}
Comparing the two formulations and the bounds \Cref{eq:bf_bound,eq:bi_bound} we see that for the inverse-lengths system the number of solutions is exponential in both the number of nodes $n$ and number of springs $s$. On the other hand for the standard formulation we see that on average the number of solutions is super-exponential in $n$ and $s$. Only in the best case, that is when $N$ is minimally rigid, is the number of solutions exponential in $n$ and $s$. Therefore, we can expect that methods using the inverse-length system will have better runtime when compared with identical methods that use the standard formulation, especially for larger spring networks. We explicitly check this for homotopy continuation on a range of spring systems in \Cref{sec:HC}.

\begin{table}[ht]
    \centering
    \begin{tabular}{|c|c|c|c|c|}
        \hline
        Graph & \begin{tabular}{c} Standard Formulation \\  Bézout bound\end{tabular}  & \begin{tabular}{c} Inverse Lengths \\ Bézout bound \end{tabular} \\ 
        \hline 
        $K_3$ &  5\,832 & 4\,096  \\ 
        \hline
        $K_4$ & 4\,193\,304 & 1\,048\,576 \\ 
        \hline
        $K_4-e$ & 663\,552 & 262\,144 \\ 
        \hline
        $K_5$ & 10\,000\,000\,000 & 1\,073\,741\,824 \\ 
        \hline
        $K_5-e$ & 2\,048\,000\,000 & 268\,435\,456 \\ 
        \hline
        $K_5-e-e$ & 419\,430\,400 & 67\,108\,864  \\ 
        \hline
        $K_5-P_3$ & 368\,640\,000 & 67\,108\,864  \\ 
        \hline
        $K_5-(P_3\cup e)$ & 75\,497\,472 & 16\,777\,216 \\ 
        \hline
        $K_5-C_3$ & 58\,320\,000 & 16\,777\,216 \\ 
        \hline
        $K_5-P_4$ & 66\,355\,200 & 16\,777\,216  \\ 
        \hline
    \end{tabular}

    \medskip
    \caption{Comparison of Bézout bound of Small 2D Spring Networks for the Standard Formulation and Inverse Lengths Formulation.}\label{tab:Bezout}
\end{table}

\subsubsection{Choice of Base Points and Bounding the Number of Homotopy Paths}\label{sec:basePoints}

Recall that in the construction of both the standard system and the inverse-length system in order to obtain finitely many solutions we fix an orientation by taking a pair of base points $p_1 = (0, 0)$ and that $p_2 = (x_2, 0)$ for some $x_2 \in \mathbb{R}$. This allows us to remove the force equations of $p_1$ in the $x$ and $y$ directions and the force equation of $p_2$ in the $y$ direction and the variables $x_1, y_1, y_2$. In fact instead of $p_1, p_2$ we can choose any pair of points $p_i, p_j$, $1 \leq i,j \leq n$ $i \neq j$ to be our base points. Thus, it is natural to ask for a given network $N = (P, S)$ there an optimal choice of base points that minimizes a bound on the maximum number of solutions. A bound on the maximum number of solutions, in turn, can be used (and is used in \texttt{hc.jl}) to reduce the number of homotopy paths that are needed to track, which is the main bottleneck in homotopy continuation.
Our result is the following:
\begin{theorem}\label{thm:bound}
    The Bézout bound for the standard formulation, $B_f$, depends on the choice of base points $p_i$,$p_j$, and $B_f$ is minimized when $p_i$ is chosen to be the node with largest degree, and $p_j$ is chosen to be the node with largest degree among $P \setminus \{ p_i \}$. The Bézout bound for the inverse-lengths formulation $B_i$ is independent of base point.
\end{theorem}
\begin{proof}
Let us first consider the standard system for a spring network $N = (P, S)$ with base points $p_i, p_j$. That is we declare that $p_i = (0, 0)$ and that $p_j = (x_2, 0)$. After removing the $3$ unnecessary equations the Bézout bound $B_f$ becomes 
\begin{equation}
    B_f = \prod_{i = 1}^{s} 2 * \frac{\prod_{k = 1}^n (1 + d(p_k))^2}{(1 + d(p_i))^2*(1 + d(p_j))} = 2^s \frac{\prod_{j = 1}^n (1 + d(p_k))^2}{(1 + d(p_i))^2*(1 + d(p_j))}.
\end{equation}

Therefore, $B_f$ is minimized when $d(p_i)$ and $d(p_j)$ are as large as possible. Explicitly choosing $p_i$ to have largest degree, and $p_j$ to have largest degree among $P \setminus \{ p_i \}$ minimizes the Bézout bound of the standard formulation. Therefore, we expect that for the standard formulation this choice of base points is optimal with respect to runtime.

Now consider the inverse-length system for a spring network $N = (P, S)$ with base points $p_i, p_j$.  Regardless of the choice of base points, after removing the $3$ unnecessary equations we obtain
\begin{equation}
    B_i = \prod_{i = 1}^{s} 4 *\frac{\prod_{j = 1}^n 2^2 }{2^2*2} 
    =2^{2s+2n - 3} 
\end{equation}
which is independent of the choice of base points.
\end{proof}

\subsubsection{Polyhedral bounds}\label{sec:polyhedral}
For polynomial systems sparsity is the concept when many fewer monomials occurring (with non-zero coefficient) than expected. Sparsity of a single polynomial can be quantified combinatorially by the Newton polytope, which is the convex hull of the exponent vectors of the monomials that occur. For a parameterized system of equations like those for spring networks we say a monomial occurs if the coefficient as a function of the parameters is not the zero function. The so-called BKK bound, \cite{Bernstein1975}, relates the number of solutions of a system with generic coefficients in the algebraic torus $(\CC^\star)^n$ to the mixed volume of the Newton polytope of the polynomials that define the system. The BKK bound for structured systems of equations can be much smaller than the Bézout bound.

It is well-known, that start systems for homotopy continuation can also be designed to match the system with respect to the mixed volume, see \cite{hauenstein2014newton} for a comprehensive overview. Moreover, \texttt{hc.jl} has this functionality built in, namely, before a homotopy continuation solving procedure begins the mixed volume is computed along with a corresponding polyhedral start system associated to the system of equations.

When a system of equations has a positive-dimensional solution set, homotopy continuation methods compute a so-called witness set, which is the solution set for the system obtained by augmenting it with codimension-many additional linear conditions. We have chosen an explicit set of linear conditions by choosing base points. This choice of base point fixes a coordinate frame, and it forces the system to have finitely many solutions. As demonstrated in \Cref{thm:bound}, the Bézout bound for the inverse-lengths system can depend on the choice of base point. In principle, we could also compute the mixed volume for each system of equations for each spring network we consider, and we could find the best choice of base point. However, this analysis would take us a bit far afield. Instead we just assume that the choice of base points that is best for the Bézout bound would also be good for the polyhedral bound. We report on the polyhedral bounds for the inverse-lengths system that \texttt{hc.jl} computes in \Cref{tab:PolyhedralSummary2D}, which supports our claim that the same choice of base points is optimal for both the Bézout and polyhedral bounds. For each spring network with $5$ or fewer nodes the choice of base point which is optimal for minimizing the Bézout bound also minimizes to polyhedral bound. In many cases the polyhedral bound is $2$ to $4$ times smaller for the optimal choice of base points when compared to the least optimal choice of base points. Additionally the polyhedral bounds for a spring network (given in \Cref{tab:PolyhedralSummary2D}) are several orders of magnitude smaller than the Bézout bound for the same spring network (given in \Cref{tab:Bezout}). 

\begin{table}[ht]
    \centering
\begin{NiceTabular}{|c|c|c|c|c|}
        \hline
        Graph & \multicolumn{4}{c}{\begin{tabular}{c} Base Point Degrees $(d(p_1), d(p_2))$ \\  Polyhedral Bound \end{tabular}} \\ 
        \hline 
        $K_3$ & \begin{tabular}{c} $(2,2)$ \\  $40$ \end{tabular}  \\ 
        \hline
        $K_4$ & \begin{tabular}{c} $(3,3)$ \\  $3\,904$ \end{tabular} \\ 
        \hline
        $K_4-e$ & \begin{tabular}{c} $(3,3)$ \\  $800$ \end{tabular} & \begin{tabular}{c} $(3,2)$ \\ $928$  \end{tabular} & \begin{tabular}{c} $(2,3)$ \\  $1\,280$ \end{tabular} & \begin{tabular}{c} $(2,2)$ \\  $1\,952$ \end{tabular} \\ 
        \hline
        $K_5$ & \begin{tabular}{c} $(4,4)$ \\  $1\,533\,952$ \end{tabular} \\ 
        \hline
        $K_5-e$ & \begin{tabular}{c} $(4,4)$ \\  $365\,568$ \end{tabular} & \begin{tabular}{c} $(4,3)$ \\   $414\,208$ \end{tabular} & \begin{tabular}{c} $(3,4)$ \\  $535\,040$ \end{tabular} & \begin{tabular}{c} $(3,3)$ \\   $766\,976$ \end{tabular} \\ 
        \hline
        $K_5-e-e$ & \begin{tabular}{c} $(4,3)$ \\  $97\,536$ \end{tabular} & \begin{tabular}{c} $(3,4)$ \\   $127\,488$ \end{tabular} & \begin{tabular}{c} $(3,3)$ \\  $144\,384$ \end{tabular}  \\ 
        \hline
        \multirow{2}{*}{$K_5-P_3$} &  \begin{tabular}{c} $(4,4)$ \\  $ 78\,080 $ \end{tabular} & \begin{tabular}{c} $(4,3)$ \\  $ 88\,832 $ \end{tabular} & \begin{tabular}{c} $(4,2)$ \\   $ 103\,168$ \end{tabular} & \begin{tabular}{c} $(3,4)$ \\  $ 120\,064$ \end{tabular} \\ \cline{2-5}  & \begin{tabular}{c} $(3,3)$ \\  $ 133\,888$ \end{tabular} & \begin{tabular}{c} $(3,2)$ \\  $ 207\,104$ \end{tabular} & \begin{tabular}{c} $(2,4)$ \\  $ 185\,600$ \end{tabular} & \begin{tabular}{c} $(2,3)$ \\  $ 267\,520$ \end{tabular}  \\ 
        \hline
        $K_5-(P_3\cup e)$ & \begin{tabular}{c} $(3,3)$ \\  $ 31\,488 $ \end{tabular} & \begin{tabular}{c} $(3,2)$ \\  $ 48\,768  $ \end{tabular} & \begin{tabular}{c} $(2,3)$ \\   $ 63\,744 $ \end{tabular} \\ 
        \hline
        $K_5-C_3$ & \begin{tabular}{c} $(4,4)$ \\  $16\,000$ \end{tabular} & \begin{tabular}{c} $(4,2)$ \\  $21\,120$ \end{tabular} & \begin{tabular}{c} $(2,4)$ \\   $39\,424$ \end{tabular} & \begin{tabular}{c} $(2,2)$ \\  $66\,944$ \end{tabular} \\ 
        \hline
        \multirow{2}{*}{$K_5-P_4$} & \begin{tabular}{c} $(4,3)$ \\  $18\,560  $ \end{tabular} & \begin{tabular}{c} $(4,2)$ \\  $ 21\,632 $ \end{tabular} & \begin{tabular}{c} $(3,4)$ \\   $ 25\,600 $ \end{tabular} & \begin{tabular}{c} $(3,3)$ \\  $ 29\,184 $ \end{tabular} \\ \cline{2-5} & \begin{tabular}{c} $(3,2)$ \\  $ 33\,280 $ \end{tabular} & \begin{tabular}{c} $(2,4)$ \\  $ 41\,600 $ \end{tabular} & \begin{tabular}{c} $(2,3)$ \\  $ 46\,336 $ \end{tabular} & \begin{tabular}{c} $(2,2)$ \\  $ 72\,192 $ \end{tabular}  \\ 
        \hline
    \end{NiceTabular}
    \medskip
    \caption{Comparison of Polyhedral Bounds for Choice of Base Points in the Inverse Lengths Formulation of Small 2D Spring Networks.} \label{tab:PolyhedralSummary2D}
\end{table}

\subsection{Newton's Method}\label{sec:newton}
One of the most standard methods for finding solutions to systems of polynomials is Newton's method. Newton's method is easy to implement, and when the initial guess is in a basin of attraction, Newton's method exhibits quadratic convergence, which makes it an attractive method.  A single instance of Newton's method, however, can only find at most a single solution and additionally may not converge at all. A standard approach is to take a lattice (or a random collection) of initial conditions and track where Newton's method sends each of them. However, as the spring networks become more complex, the space of initial conditions that needs to be checked becomes large since each additional variable increases the dimension of the sample space. This makes it difficult to sample the space of large spring networks with sufficient density. Furthermore, it appears that the basin of attraction around each solution becomes very small as the polynomial system grows. This makes it increasingly unlikely for each individual initial condition to converge via Newton's method. Thus, increasing numbers of trials is needed to find even one solution to larger systems using Newton's method. 

\begin{remark}The question of determining the characteristics of the grid (size, density, etc.) of initial conditions for Newton solving is at the heart of the famous P vs NP problem.  One prototypical NP hard problem is quadratic feasibility, that is determining if a system of quadratic equations has solutions over $\QQ, \RR$ or $\CC$. One might attempt to solve this with Newton's method with a grid of initial conditions and wait to see if any initial point converges to a solution. Even though each instance of Newton's method takes polynomial time, the efficiency of solving depends on the number of initial points that must be considered. The P vs  NP question is if a provably sufficient initial grid must have the number of initial grid points $N(n)$ exponential in $n$ the number of equations. Random choices of grid points can help in practice to keep $N(n)$ polynomially small for most systems, and de-randomization strategies are one approach to resolving P vs NP. For more about P vs NP related to linear and multilinear algebra, see \cite{HillarLim09mosttensor}.
\end{remark}

We assume our system of equations has solutions and that by taking enough random initial points we will eventually converge to a solution. We can heuristically quantify the concept ``eventually'' by computing the average time it takes Newton's method to find a single solution to each system and the number of initial points (trials) it takes to find one which converges to a solution. All computations reported in this section were performed using the Julia library \texttt{hc.jl} on Auburn University's College of Sciences and Mathematics (COSAM) server (Saturn), which has 96 cores.

\begin{table}[ht]
    \centering
    \scalebox{.95}{\begin{tabular}{|c|c|c|c|}
        \hline
        Graph &  Number of Trials &  \begin{tabular}{c} Newton's Method \\ Average Time \end{tabular}  & \begin{tabular}{c}Additional \\ Monodromy \\Time \end{tabular} \\
        \hline 
        $K_3$ &  \begin{tabular}{c} 4.42 \\33.56  \\521.24 \end{tabular} & \begin{tabular}{c} 104.8 $\mu$s \\ 225.4 $\mu$s \\ 2.29 ms \end{tabular} & 1.67s \\
        \hline
        $K_4$ & \begin{tabular}{c} 63.39 \\ 1\,649 \\ 47\,393 \end{tabular} & \begin{tabular}{c} 898.3 $\mu$s \\ 16.12 ms \\ 403.0 ms \end{tabular} & 11.1s \\
        \hline
        $K_4-e$ & \begin{tabular}{c} 19.78\\331.3\\10\,627 \end{tabular} & \begin{tabular}{c} 333.3 $\mu$s\\ 3.22 ms \\ 90.86 ms \end{tabular} & 8.02s \\
        \hline
        $K_5$ &  \begin{tabular}{c} 1\,842 \\ 70\,443 \\ 9\,174\,311 \end{tabular} & \begin{tabular}{c} 45.16 ms \\ 893.5 ms \\ 295.1 s \end{tabular} & n/a \\
        \hline
        $K_5-e$ & \begin{tabular}{c} 982.6 \\ 41\,452\\ 7\,142\,857 \end{tabular}& \begin{tabular}{c}19.80 ms \\ 889.5 ms \\ 97.2 s \end{tabular} &  26m 26s  \\
        \hline
        $K_5-e-e$ & \begin{tabular}{c} 25\,062 \\ 96\,153 \\ 21\,276\,595 \end{tabular} & \begin{tabular}{c} 462.7 ms \\ 2.62 s \\ 585.7 s \end{tabular} &  12m 56s  \\
        \hline
        $K_5-P_3$ & \begin{tabular}{c} 1392 \\ 10\,741 \\ 1\,369\,863 \end{tabular} & \begin{tabular}{c} 25.61 ms \\ 218 ms \\ 25.1 s \end{tabular} & 6m 29s \\
        \hline
        $K_5-(P_3\cup e)$ & \begin{tabular}{c} 248.1 \\ 15\,432 \\ 285\,714 \end{tabular} & \begin{tabular}{c} 4.41 ms \\ 284.2 ms \\ 2.89 s \end{tabular} & 1m 36s \\
        \hline
        $K_5-C_3$ & \begin{tabular}{c} 39.51 \\ 2\,567 \\ 104\,166 \end{tabular} & \begin{tabular}{c} 836.3 $\mu$s \\ 40.64 ms \\ 1.30 s \end{tabular} & 12.8s \\
        \hline
        $K_5-P_4$ & \begin{tabular}{c} 221.6 \\ 3\,828 \\ 88\,495 \end{tabular} & \begin{tabular}{c} 3.57 ms \\ 60.53 ms \\ 3.63 s \end{tabular} & 23.1s \\
        \hline
    \end{tabular}}

    \medskip
    \caption{Computational Results for Small 2D Spring Networks with Typical Parameter Values. For each graph we sample initial conditions in $3$ regions: $[0,1]^n$, $[-1,1]^n$ and $[-10, 10]^n$. We record the average number of trials until a successful convergence of Newton's method, the average time to obtain one solution, and the additional time needed to find all solutions with monodromy. Cells with n/a represents computations that were terminated after failing to conclude after 4 hours.}\label{tab:Newton}
\end{table}

One consideration when using Newton's method is that we do not a priori know in what domain our solutions will lie. In order to begin a Newton's method computation we must first choose a region of our input space to take a lattice in (or randomly select points from). By Rouche's Theorem we can expect that all solutions lie in some ball $B_{(r, 0)}$ centered at the origin, but the radius $r$ depends strongly on the choice of parameters and is difficult to compute for large systems. Furthermore we expect this to be a weak bound. As shown in the \Cref{tab:Newton} for each spring network, sampling a larger region significantly increases the time it takes to find a single solution using Newton's method as well as the average number of trials until Newton's method converges. This effect is especially pronounced for larger spring networks. In fact for spring networks $K_5-e$ and $K_5-e-e$ when sampling in the region $[-10, 10]^n$ it takes longer, on average, for Newton's method to find a single solution, than homotopy continuation to find all solutions to the system. 

Furthermore, Newton's method finds only one solution to the spring system. Thus, in order to find all solutions using Newton's method we have two options. Firstly, we could continually try Newton's method using random initial starting points until the number of distinct solutions found stabilizes. However, for large spring systems both the total number of solutions and the time it takes to find a single solution using Newton's method grow, so this becomes increasingly computationally inefficient. Furthermore, this process has no way to check when all solutions have been found. We anticipate that any method that finds one solution at a time will suffer similarly. 

Alternatively we could use the one solution given by Newton's method to begin a monodromy calculation, which we will describe in \Cref{sec:monodromy}. Using monodromy is significantly faster computationally and, while it still has no guarantee of finding all solutions, there are heuristic stopping criteria that give reasonable confidence that all solutions are found. We report the total time it takes a combined Newton and monodromy calculation to find all solutions in \Cref{tab:Newton}.

\subsection{Homotopy Continuation}\label{sec:HC}
The main technique we utilized for solving these systems of polynomials is homotopy continuation.

Given an unknown polynomial system $F(\overline{x})$, one first constructs a start system $G(\overline {x})$ with at least as many roots as $F(\overline{x})$ and whose roots are known. One popular choice for $G(\overline{x})$ is a system with roots being roots of unity for some high degree. We then consider the homotopy from $G(\overline{x})$ to $F(\overline{x})$, that is the function $H(\overline{x}, t) = (1 - t) G(\overline{x}) + t F(\overline{x})$. We track the path of each known solution of $G$ to a solution of $F$ as $t$ goes from $0$ to $1$.

The homotopy continuation method is effective in solving modest sized spring networks. Homotopy continuation solvers must first create the start system $G(\overline{x})$. The degree of the start system depends on the total degree of the system that we want to solve. This increases as more nodes and springs are added to the spring network. Start systems with large degree require significant computations. The initialization of the start system is only required the first time solving a new spring network, but it does introduce additional computations. We therefore examine alternate methods for solving large spring systems that reduce computation time in \Cref{sec:big}.

\subsection{Monodromy}\label{sec:monodromy}
Monodromy has been shown in \cite{BASKAR2024105609} to be an effective method for dealing with systems of equations that involve non-algebraic functions, which, in particular, make it unclear what start system should be chosen for homotopy methods. In addition, these methods can be useful when the systems of equations have a large difference between the number of paths tracked and the number of non-singular solutions found. This approach was also described in \cite{TARI2012177}.

Monodromy has a range of benefits and limitations when compared with other methods for solving our systems of equations. It is often useful as a supplementary tool to homotopy continuation. Firstly monodromy requires an initial solution of the system to get started. This means that we must first use a different strategy such as Newton's method, homotopy continuation or parameter/cheater homotopy to find an initial solution to our system of equations that we can use in monodromy.

\subsubsection{Heuristic Stopping Criteria for Finding all Solutions}
Given an initial solution a monodromy loop will take that solution to another solution, and initially it is likely that the new solution is distinct from the previous one. In subsequent loops each solution in the current solution set is sent to another solution potentially doubling the total number of solutions found if there are no repeated solutions. In practice we observe that initially the number of solutions approximately doubles for each monodromy loop until it starts approaching the full solution set, in which case the monodromy loop will mostly take the current solution set to itself, finding fewer new solutions at each loop. This means that if monodromy is initialized with a few solutions to the system of equations it can rapidly find a large percentage of total number of solutions to the system. A heuristic stopping criterion is when several monodromy loops in a row find no new solutions.

For very large systems of equations arising from large spring networks numerical errors can cause homotopy continuation to miss a small number of solutions. In this case monodromy is useful in finding the missing solutions. Relatedly, monodromy is useful in verifying (heuristically) that another method has identified all of the solutions to a system of equations.  

\subsubsection{Missing Components}
One limitation of monodromy is that, if it initialized with a single solution or very few solutions, monodromy will, on occasion, miss the vast majority of solutions. This occurs when the incidence variety (the zero-set of the system of equations where all parameters are also treated as variables) is disconnected. Given an initial solution, monodromy can only find solutions that are on the same connected component of the incidence variety. In order to avoid this issue we can modify the system by introducing a set of `dummy' variables. This ensures that our system will have an incidence variety consisting of a single connected component, which allows us to consistently find all solutions using monodromy given any set of initial solutions.

\subsection{Parameter Homotopy and Building Up Networks}\label{sec:ph}

Computing the full solution set for a spring network using homotopy continuation becomes increasingly difficult and computationally expensive as the numbers of nodes and edges of the network grow. We can attempt to find solutions to larger networks from the solutions for smaller networks by adding in new springs (and possibly nodes) that are initially at rest, and hence do not introduce new solutions, and then performing a homotopy to move these solutions to the new system where those springs are not at rest.

Consider a spring network $N = (P, S)$. A \emph{subnetwork} of $N$ is a network $N' = (P', S')$ such that $P' \subseteq P$ and $S' \subseteq S$. Let $F(\overline{x},\overline{p})$ be the polynomial system that represents network $N$, where $\overline{x}$ are the variables of $N$ and $\overline{p}$ represent the parameters of $N$. Notice that we can describe the polynomial system of $N'$ simply by changing the parameter values $\overline{p}$ of $F$. Specifically for any $s_{i, j} \in S$, if $s_{i,j } \in S \setminus S'$ we set the parameter $k'_{i,j} = 0$ and if $s_{i,j} \in S'$ and we set the parameter $k'_{i,j} = k_{i,j}$. In both cases we leave ${\tilde{\ell}}'_{i,j} = {\tilde{\ell}}_{i,j}$. This gives a system identical to polynomial system representing $N'$ with the addition of dummy variables $I_{i,j}$ for $s_{i,j } \in S \setminus S'$ and $x_i, y_i$ for $p_i \in P \setminus P'$, and geometric equations that give their value in terms of the $x_i$'s and $y_i$'s. Let $\overline{q}$ be the parameters for the system that represents $N'$.

To conduct a parameter homotopy we first solve the polynomial system $F(\overline{x},\overline{q})$. This can be done using homotopy continuation or monodromy. We then take a homotopy from the system representing $N'$ to the system representing $N$ by changing the parameters. The homotopy is the function $H(\overline{x}, t) = F(\overline{x}, (1- t) \overline{q} + t \overline{p})$. We then track the path of each solution of $F(\overline{x},\overline{q})$ to a solution of $F(\overline{x},\overline{p})$ as $t$ goes from $0$ to $1$, in the same manner as homotopy continuation. 

Compared to $F(\overline{x},\overline{p})$, many of the parameters in $F(\overline{x},\overline{q})$ are $0$ which causes terms in various equations to disappear. This lowers the overall degree of the system meaning that the system $F(\overline{x},\overline{q})$ is faster to solve when compared to directly solving $F(\overline{x},\overline{p})$. Additionally, for large systems, the time it takes to track the paths through the parameter homotopy is negligible when compared to the time to compute $F(\overline{x},\overline{p})$. This makes parameter homotopy a significantly faster technique for solving large spring networks.

One drawback of this technique is that parameter homotopy will inevitably miss solutions to the larger system $F(\overline{x},\overline{p})$. Parameter homotopy tracks the path of each solution of the smaller system $F(\overline{x},\overline{q})$ and so it can at most produce a number of solutions equal to the number of solutions of $F(\overline{x},\overline{q})$. This is generally fewer than the total number of solutions of $F(\overline{x},\overline{p})$.

As a result we often use parameter homotopy in conjunction with monodromy since it is able to quickly find some initial solutions to the larger system that can be used as a basis for the monodromy loops. Since parameter homotopy often finds many solutions to the larger system which can help to accelerate the monodromy computation. 

\begin{example}\label{ex:Parameter Homotopy}

Consider two networks $N$ and $N'$, depicted respectfully in  \Cref{fig:K5network}. Network $N'$ is a subnetwork of $N$. 

Using homotopy continuation we can solve the network $N'$. Note that viewing $N'$ as a subnetwork of $N$ requires us to add dummy variables for the springs that are present in $N$ but not $N'$ and geometric equations that govern their length. This slightly increases the complexity of the computation time for $N'$ when compared to is computation depicted in \Cref{sec:HC}. Solving $N'$ involves tracking $731\,136$ paths, less than half of the paths required to directly solve the system $N$ directly using homotopy continuation. Furthermore the time of computation is only $13$:$47$, when compared with the $1$:$05$:$47$ it takes to solve $N$ directly. The computation for $N'$ finds $27\,246$ non-singular solutions, $3\,070$ of which are real.

We now take a parameter homotopy from $N'$ to $N$. This tracks each solution of $N'$ as it traces a path to a solution of $N$. The time it takes to track the paths of all solution is only $23$s. As some distinct solutions of $N'$ map to the same solution of $N$ and some paths diverge, this results in fewer total solutions than were computed for the network $N'$. In particular it results in $16\,334$ total solutions: $19$ singular and $16\,312$ non-singular. $964$ of those solutions are real, of which $11$ are singular and $953$ are non-singular.  

This method gives roughly a quarter of all solutions of network $N$, in a fraction of the time. The remaining solutions can be found using monodromy methods. Using monodromy it takes an additional $24$:$43$ to find the remaining solutions to the system. The total time to solve $N$ by this hybrid method is 
$38$:$53$. This is a 1.7x speedup compared to the $1$:$05$:$47$ needed for standard homotopy continuation.
\end{example}

\section{Results: Solving 2- and 3-Dimensional Spring Networks}\label{sec:2d}
A $2$-dimensional linear spring network is comprised a collection of
points in $\mathbb{R}^2$ connected by a network of linear
springs with known resting length and spring constants. 

We compute the equilibria for all rigid $2$-dimensional linear spring networks with 3 to 5 nodes for typical initial parameter values. We give the number of both singular and nonsingular solutions. Additionally, we construct an algorithm to compute equilibria for general $2$-dimensional linear spring systems with more than $5$ nodes in \Cref{sec:big}. All computations reported in this section were performed using \texttt{hc.jl} on Auburn University COSAM Saturn using 96 cores.

\Cref{tab:NetworkSummary2D} gives a summary of the results of our computations for the  parameter values provided in the examples in this section. We note that the numbers of paths and solutions depend on the initial parameters, however, the orders of magnitude of these quantities seem to be constant over a wide range of valid parameter values.
\begin{table}[ht]
    \centering\scalebox{.9}{
    \begin{tabular}{|c|r|r|r|r|c|c|c|c}
    \hline
        Graph & \#paths & 
        \begin{tabular}{c}
        \#solutions  \\ 
        complex  (real) \end{tabular} &
        \begin{tabular}{c}
        \# non-singular  \\ 
        complex (real) \end{tabular} & \begin{tabular}{c}
            time \\ (h:m:s)  
        \end{tabular}
        \\
        \hline 
        $K_3$ & \begin{tabular}{r} 228 \\ 40\end{tabular} &  
        \begin{tabular}{r} 31 (22) \\ 12 (12)\end{tabular}  &  
        \begin{tabular}{r} 12(12) \\ 12(12)\end{tabular} & 
        \begin{tabular}{l} 0:00:07.74 \\ 0:00:07.34\end{tabular} \\
        \hline
        $K_4$ &\begin{tabular}{r} 80\,136 \\ 3\,904 \end{tabular}& 
        \begin{tabular}{r} 4\,147 (440)\\ 424 (148)\end{tabular}  & 
        \begin{tabular}{r} 424 (148)\\ 424 (148)\end{tabular} & 
        \begin{tabular}{l} 0:01:55 \\ 0:00:13.1 \end{tabular} \\
        \hline
        $K_4-e$ &\begin{tabular}{r} 16\,416 \\ 1280 \end{tabular} & 
        \begin{tabular}{r} 928 (240)\\ 72 (72)\end{tabular}  & 
        \begin{tabular}{r} 72 (72)\\ 72 (72)\end{tabular} & 
        \begin{tabular}{l} 0:00:26.2 \\ 0:00:09.92\end{tabular}\\
        \hline
        $K_5$ &\begin{tabular}{r} n/a \\ 1\,533\,952 \end{tabular} & 
        \begin{tabular}{r} n/a \\ 53\,377 (4\,176) \end{tabular} &
        \begin{tabular}{r} n/a \\ 53\,338 (4\,176) \end{tabular} &
        \begin{tabular}{l} n/a \\ 1:05:47 \end{tabular} \\
        \hline
        $K_5-e$ &\begin{tabular}{r} n/a \\ 365\,568 \end{tabular} & 
        \begin{tabular}{r} n/a \\ 16\,632 (1\;846)  \end{tabular} &
        \begin{tabular}{r} n/a \\ 16\,607 (1\,846) \end{tabular} &
        \begin{tabular}{l} n/a \\ 0:11:51 \end{tabular} \\
        \hline
        $K_5-e-e$ &\begin{tabular}{r} n/a \\ 144\,384 \end{tabular} & 
        \begin{tabular}{r} n/a \\ 4\,556 (831) \end{tabular} &
        \begin{tabular}{r} n/a \\ 4\,556 (831) \end{tabular} &
        \begin{tabular}{l} n/a \\ 0:04:37 \end{tabular} \\
        \hline
        $K_5-P_3$ &\begin{tabular}{r} n/a \\ 133\,888 \end{tabular} & 
        \begin{tabular}{r} n/a \\ 2\,498 (892) \end{tabular} &
        \begin{tabular}{r} n/a \\ 2\,496 (892) \end{tabular} &
        \begin{tabular}{l} n/a \\ 0:03:56 \end{tabular} \\
        \hline
        $K_5-(P_3\cup e)$ &\begin{tabular}{r} 983\,140 \\ 31\,488 \end{tabular} &
        \begin{tabular}{r} 47\,109 (1\,924) \\ 695 (268) \end{tabular} &
        \begin{tabular}{r} 592 (248) \\ 592 (248) \end{tabular} &
        \begin{tabular}{l} 0:34:52 \\ 0:00:50 \end{tabular} \\
        \hline
        $K_5-C_3$ &\begin{tabular}{r} 756\,752 \\ 16\,000 \end{tabular} & 
        \begin{tabular}{r} 13\,680 (1\,927) \\ 412 (284) \end{tabular} &
        \begin{tabular}{r} 412 (284) \\ 412 (284) \end{tabular} &
        \begin{tabular}{l} 0:18:13 \\ 0:00:25 \end{tabular} \\
        \hline
        $K_5-P_4$  &\begin{tabular}{r} 1\,216\,694 \\ 29\,184 \end{tabular} &
        \begin{tabular}{r} 24\,434 (2\,377) \\ 397 (393) \end{tabular} &
        \begin{tabular}{r} 397 (393) \\ 397 (393) \end{tabular} &
        \begin{tabular}{l} 0:41:25 \\ 0:00:45 \end{tabular} \\
    \hline \end{tabular} }\medskip
    \caption{Computational Results for Small 2D Spring Networks with Typical Parameter Values. For each graph the results reported in the top line use the length formulation, the second line are for the inverse-length formulation. Rows with ``n/a'' are computations we did not attempt. }\label{tab:NetworkSummary2D}
\end{table}

\subsection{3-Node Spring Networks}

Consider a spring network $N$ with $3$ nodes and $3$ springs. That is $P = (p_1, p_2, p_3)$ and $S = (s_{1,2}, s_{1,3}, s_{2,3})$. This gives a system of $6$ polynomial equations with $6$ parameters $k_{1,2}, k_{1,3}, k_{2,3}, \tilde{\ell}_{1,2}, \tilde{\ell}_{1,3}, \tilde{\ell}_{2,3}$.

\begin{figure}[ht] \centering 
\includegraphics[scale=.4]{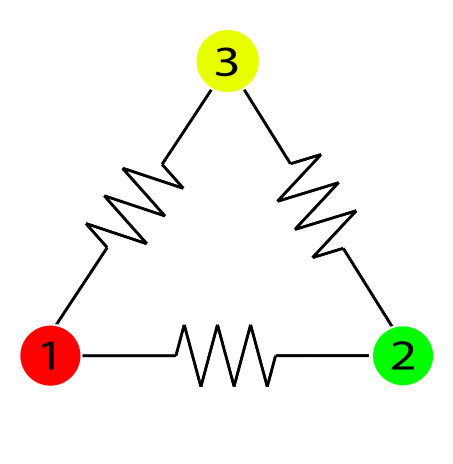}
\vspace{-2em}
\caption{$K_3$ Spring Network. We take typical values $k_{1,2} = k_{1,3} = k_{2,3} = 0.5$, $\tilde{\ell}_{1,2} = \tilde{\ell}_{1,3} = 1$, and $\tilde{\ell}_{2,3} = 1.5$.}\label{fig:K3network}
\end{figure}
    
\begin{example}\label{ex:K3}
Consider an example system with the parameter values $k_{1,2} = k_{1,3} = k_{2,3} = 0.5$, $\tilde{\ell}_{1,2} = \tilde{\ell}_{1,3} = 1$, and $\tilde{\ell}_{2,3} = 1.5$. Solving the system given by inverse spring lengths using homotopy continuation gives 12 real, non-singular solutions. The initial computation takes approximately 7.34s. 

 We graph each of the twelve solutions, labeled by their potential energy (PE) in \Cref{fig:K3sols}, from which we see that several solutions are identical up to symmetry. The variables values for the 12 real nonsingular solutions rounded to 3 decimal places are given in the matrix  below.
\[ 
\left[ \begin{smallmatrix}
    x_2 & x_3 & y_3 & I_{1,2} & I_{1,3} & I_{2,3} \\
    \hline\\
  1.0& -0.125& 0.992& 1.0& 1.0& 0.667 \\
  1.0& -0.125& -0.992& 1.0& 1.0& 0.667 \\
  -1.0& 0.125& -0.992& 1.0& 1.0& 0.667 \\
  -1.0& 0.125& 0.992& 1.0& 1.0& 0.667 \\
  -0.833& 0.833& 0.0& 1.2& 1.2& 0.6 \\
  0.833& -0.833& 0.0& 1.2& 1.2& 0.6 \\
  1.5& 0.5& 0.0& 0.667& 2.0& 1.0 \\
  -0.5 & -1.5 & 0.0& 2.0& 0.667& 1.0\\
   0.5& 1.5& 0.0& 2.0& 0.667& 1.0 \\
  -1.5& -0.5& 0.0& 0.667& 2.0& 1.0 \\
  0.167& -0.167& 0.0& -6.0& -6.0& 3.0 \\
  -0.167& 0.167& 0.0& -6.0& -6.0& 3.0 
  \vrule width 0pt depth 2pt
\end{smallmatrix} \right]
\]

\begin{figure}[ht] \centering \label{fig:K3sols}
\includegraphics[scale=.62]{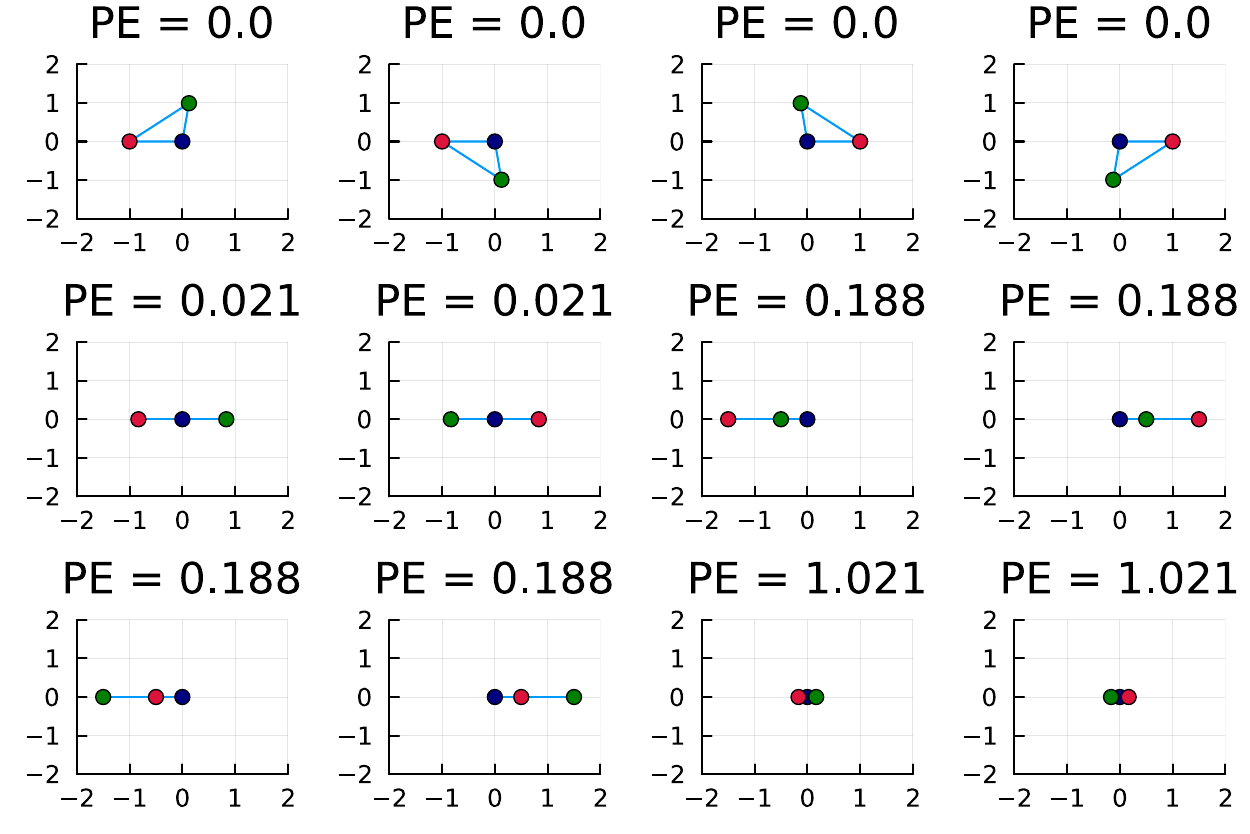}
\caption{Plot of the Real, Nonsingular Solutions for the K3 Spring Network Described in \Cref{ex:K3}.}\label{fig:K3solutions}
\end{figure}

Solving the system given by normal spring lengths also gives 12 real, non-singular solutions. These solutions are identical to the solutions given by the previous method. However, it also gives 19 additional singular solutions, 11 of which are real. The initial computation takes approximately 7.74s. 

The variable values for the singular solutions rounded to 3 decimal places are given in the matrix below. The real solutions are given on the left matrix, and the complex solutions are given in the right matrix:
\[
\left[
\begin{smallmatrix} \vrule width 0pt depth 4pt
  x_2 & x_3 & y_3 & \ell_{1,2} & \ell_{1,3} & \ell_{2,3} \\
  \hline \\
  0.0& 1.25& 0.0& 0.0& 1.25& 1.25 \\
  0.0& 0.0& 0.0& 0.0& 0.0& 0.0 \\
  0.25& 0.0& 0.0& -0.25& 0.0& 0.25 \\
  0.25& 0.0& 0.0& -0.25& 0.0& 0.25 \\
  1.25& 0.0& 0.0& 1.25& 0.0& 1.25 \\
  0.0& -1.25& 0.0& 0.0& 1.25& 1.25 \\
  1.0& 1.0& 0.0& 1.0& 1.0& 0.0 \\
  0.0& 0.25& 0.0& 0.0& -0.25& 0.25 \\
  0.0& -0.25& 0.0& 0.0& -0.25& 0.25 \\
  -1.25& 0.0& 0.0& 1.25& 0.0& 1.25 \\
 - 1.0& -1.0& 0.0& 1.0& 1.0& 0.0 
  \vrule width 0pt depth 2pt
\end{smallmatrix}\right], \quad 
      \left[\begin{smallmatrix}
  x_2 & x_3 & y_3 & \ell_{1,2} & \ell_{1,3} & \ell_{2,3}\\
\hline \\
0.0 & 0.506 - 3.376 i & 3.385 + 0.505 i & 0.0& -0.25 & 0.25 \\
0.0 & 0.016 + 0.218 i & 1.269 - 0.003 i & 0.0 & 1.25& 1.25\\
0.0 & 0.135 - 0.047 i & 0.218 + 0.029 i & 0.0 & -0.25 & 0.25 \\
0.0 & -0.652 - 0.204 i & -1.093 + 0.122 i & 0.0 & 1.25 & 1.25 \\
0.0 & 1.049 + 0.046 i & 0.685 - 0.070 i & 0.0 & 1.25& 1.25 \\
0.0 & -7.746 + 4.652 i & 4.654 + 7.743 i & 0.0 & -0.25 & 0.245 \\
0.0 & -0.175 - 0.299 i & -1.274 + 0.041 i & 0.0 & 1.25 & 1.25 \\
0.0 & -0.123 + 0.051 i & -0.225 - 0.02 i & 0.0 & -0.25 & 0.25 \\
      \end{smallmatrix}\right].\qedhere
      \]
\end{example}

\subsection{4-Node Spring Networks}
In this section we list as figures the 4-node spring networks we study and report on in \Cref{tab:NetworkSummary2D}. In the captions we list the typical parameter values we choose for our calculations.  

\begin{figure}[ht] \centering 
\includegraphics[scale=.35]{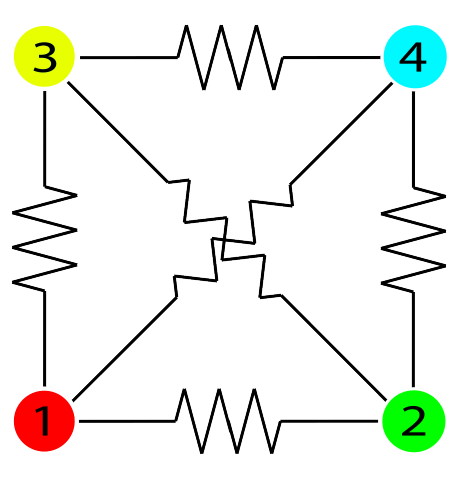}
\quad
\includegraphics[scale=.5]{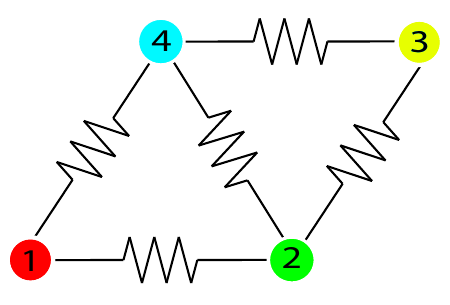 }
\caption{Spring Networks $K_4$ (left) and $K_4 - e$ (right). We take typical parameter values $k_{i,j} = 0.5$ for both, and $\tilde{\ell}_{1,2} = \tilde{\ell}_{2,3} = \tilde{\ell}_{3,4} = \tilde{\ell}_{1,4} = 1$, $\tilde{\ell}_{1,3} = \tilde{\ell}_{2,4} = \sqrt{2}$ for $K_4$ and $\tilde{\ell}_{1,2} = \tilde{\ell}_{2,3} = \tilde{\ell}_{3,4} = \tilde{\ell}_{1,4} = \tilde{\ell}_{2,4} = 1$ for $K_4-e$.}\label{fig:K4network}\label{fig:K4-enetwork}
\end{figure}

Note that for 4 nodes we have listed the only 2-d planar rigid graphs. When moving to 3-d, the $K_4$ network can still be rigid, and (as long as the lengths satisfy all relevant triangle inequalities) produces solutions that are a tetrahedron with all edges at rest, see \Cref{fig:K4-3Dnetwork}. The planar solutions are then singular solutions when the system is viewed in 3-d.  The $K_4-e$ spring network is not 3-d rigid, on the other hand.

\subsection{5-Node Spring Networks}
Here we report on solutions to 5-node spring networks in 2- and 3- dimensions, and because the systems become more challenging to solve as the network grows we discuss strategies for solving the larger networks among these examples.

\begin{figure}[ht] \centering 
\includegraphics[scale=.5]{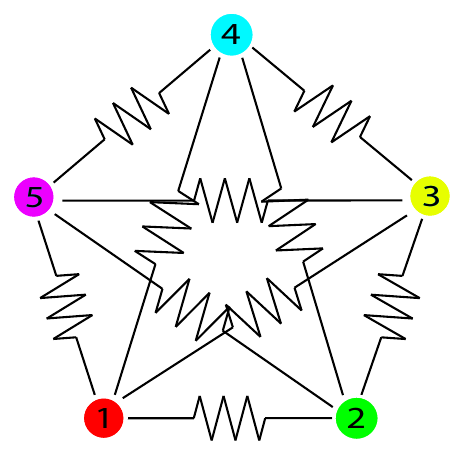}
\qquad
\includegraphics[scale=.5]{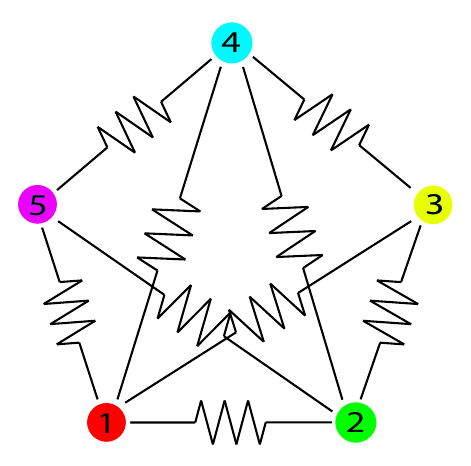 }
\caption{Spring Networks $K_5$ (left)  and $K_5 - e$ (right). We take typical the parameter values $k_{1,2} = k_{2,3} = k_{3,4} = k_{4,5} = k_{1,5} =  1$, $k_{1,3} = k_{1,4} = k_{2,4} = k_{2,5} = k_{3,5} = 0.5$, $\tilde{\ell}_{1,2} = \tilde{\ell}_{2,3} = \tilde{\ell}_{3,4} = \tilde{\ell}_{4,5} = \tilde{\ell}_{1,5} = 1$, and $\tilde{\ell}_{1,3} = \tilde{\ell}_{1,4} = \tilde{\ell}_{2,4} = \tilde{\ell}_{2,5} = \tilde{\ell}_{3,5} = \frac{1 + \sqrt{5}}{2}$ for $K_5$ and similarly for $K_5-e$ except that $k_{3,5}$ and $\tilde\ell_{3,5}$ are missing in that case.}\label{fig:K5network}
\end{figure}

\begin{figure}[ht] \centering 
\includegraphics[scale=.4]{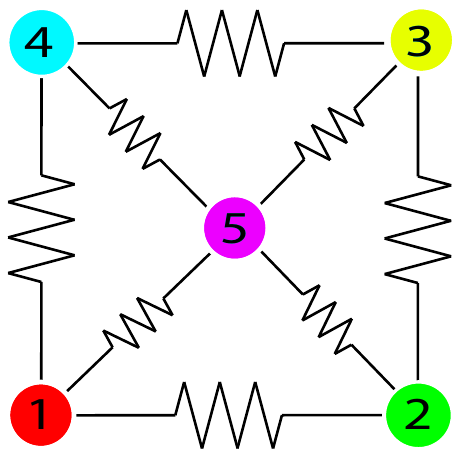}
\qquad\includegraphics[scale=.5]{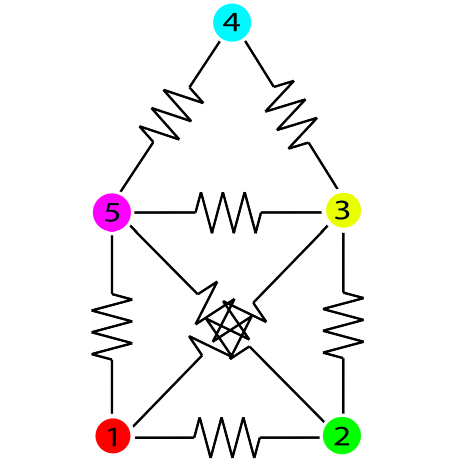 }
\caption{Spring Networks $K_5 - e - e$ (left) and $K_5 - P_3$ (right). For $K_5 - e - e$ we take typical parameter values $k_{1,2} = k_{2,3} = k_{3,4} = k_{1,4} = 1$, $k_{1,5} = k_{2,5} = k_{3,5} = k_{4,5} = 0.5$, $\tilde{\ell}_{1,2} = \tilde{\ell}_{2,3} = \tilde{\ell}_{3,4} = \tilde{\ell}_{1,4} = 2$, and $\tilde{\ell}_{1,5} = \tilde{\ell}_{2,5} = \tilde{\ell}_{3,5} = \tilde{\ell}_{4,5} = \sqrt{2}$.
For $K_5 - P_3$ we take typical parameter values $k_{1,2} = k_{2,3} = k_{3,5} = k_{1,5} = 1$, $k_{1,3} = k_{2,5} = 0.5$, $k_{3,4} = k_{4,5} = 1.5$ $\tilde{\ell}_{1,2} = \tilde{\ell}_{2,3} = \tilde{\ell}_{3,5} = \tilde{\ell}_{1,5} = 1$, $\tilde{\ell}_{1,3} = \tilde{\ell}_{2,5} = \sqrt{2}$, $\tilde{\ell}_{3,4} = \tilde{\ell}_{4,5} = 2$.}
\label{fig:K5-e-enetwork}
\end{figure}

\begin{figure}[ht] \centering 
\includegraphics[scale=.5]{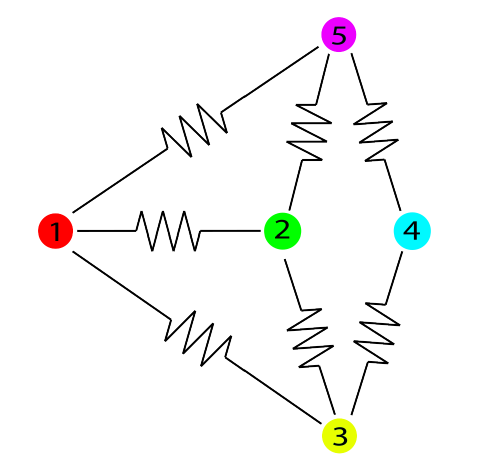}
 \includegraphics[scale=.5]{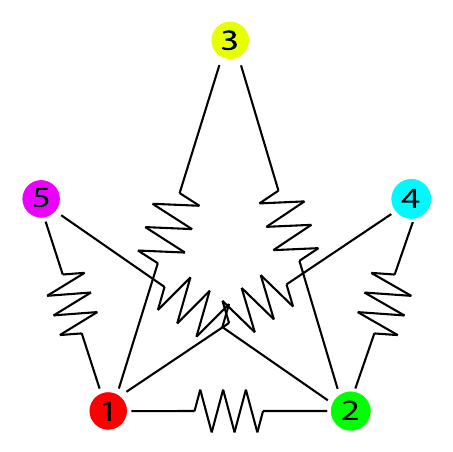 }
\includegraphics[scale=.5]{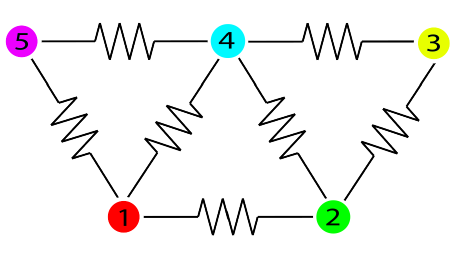 }
\caption{Spring Networks $K_5 - (P_3 \cup e)$ (left), $K_5 - C_3$ (center), and $K_5 - P_4$ (right). For $K_5 - (P_3 \cup e)$ we take typical parameter values $k_{1,2} = 1$, $k_{2,3} = k_{3,4} = k_{4,5} = k_{2,5} = 1.5$, $k_{1,3} = k_{1,5} = 0.5$, $\tilde{l}_{1,2} = 1$, $\tilde{l}_{2,3} = \tilde{l}_{3,4} = \tilde{l}_{4,5} = \tilde{l}_{2,5} = 0.5$, $\tilde{l}_{1,3} = \tilde{l}_{1,5} = 1.5$.
For $K_5 - C_3$ we take typical parameter values $k_{1,2} = k_{1,5} = k_{2,5} = 1$, $k_{1,3} = k_{2,3} = k_{1,4} = k_{2,4} = 0.5$, $\tilde{\ell}_{1,2} = \tilde{\ell}_{1,5} = \tilde{\ell}_{2,5} = 1$, $\tilde{\ell}_{2,3} = \tilde{\ell}_{1,4} = 0.5$, and $\tilde{\ell}_{1,3} = \tilde{\ell}_{2,4} = 1.2$. For $K_5 - P_4$ we take typical parameter values $k_{i,j}= 0.5$, $\tilde \ell_{i,j}= 1$.
}\label{fig:K5-P3-enetwork}
\end{figure}

\subsection{Solving 3-Dimensional Linear Spring Networks}\label{sec:3d}

A $3$-dimensional Linear Spring Network is comprised a collection of points in $\mathbb{R}^3$. These points are connected by a network of linear springs with known resting lengths and spring constants.

\Cref{tab:NetworkSummary3D} gives a summary of the results of our computations for typical parameter values. The figures that follow refer to this table and the typical parameter values we chose are listed in their captions. 
\begin{table}[ht]
    \centering

    \begin{tabular}{|c|c|c|c|c|c|c|c|c|}
        \hline
        Graph & \#Paths & \begin{tabular}{c}\#Solutions  \\ Complex (Real) \end{tabular} & \begin{tabular}{c}\#Solutions \\ Non-singular  \\ Complex (Real) \end{tabular} & \begin{tabular}{c}Wall \\ Time \end{tabular} \\
        \hline 
        $K_4$ & 5\,504 & 248 (128)  & 88 (88) & 0:00:27.7 \\
        $K_5$ & 7\,993\,344 & 439\,076 (8\,728) & 438\,553 (8\,728) & 4:25:44 \\
        $K_5-e$ & 1\,544\,704 & 57\,119 (5\,412) & 56\,169 (5\,172) & 0:34:36 \\
        \hline
    \end{tabular}

    \caption{Computational Results for Small 3D Spring Networks (using inverse-length formulations).}\label{tab:NetworkSummary3D}
\end{table}
 
\begin{figure}[ht] \centering 
\includegraphics[scale=.6]{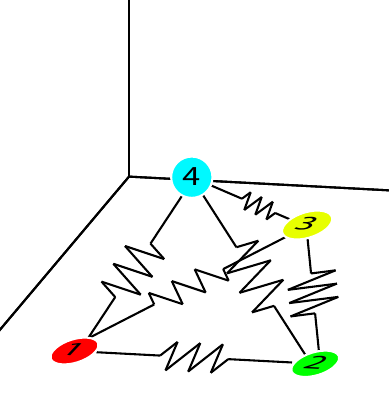}
 \includegraphics[scale=.6]{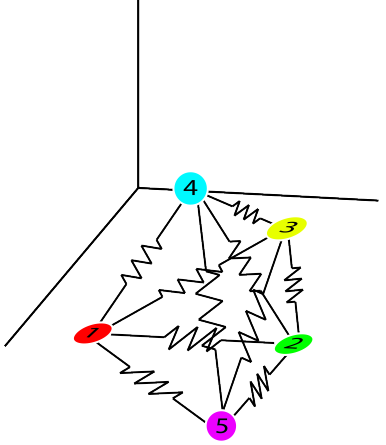}  
 \includegraphics[scale=.6]{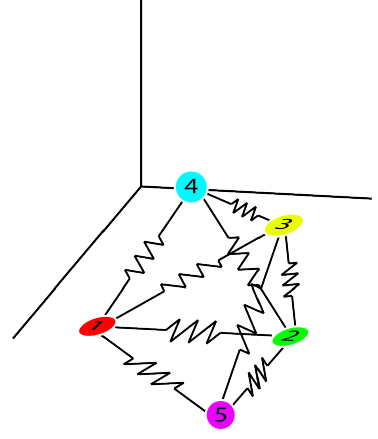}
\caption{$3$-Dimensional Spring Networks $K_4$ (left) and $K_5$ (center) and $K_5-e$ (right). For $K_4$ we take typical parameter values $k_{i,j} = 0.5$, $\tilde{\ell}_{1,2} = \tilde{\ell}_{2,3} = \tilde{\ell}_{3,4} = \tilde{\ell}_{1,4} =\tilde{\ell}_{1,3} = \tilde{\ell}_{2,4} = 1$.
For $K_5$ we take typical parameter values $k_{i,j} = 1$, $\tilde{\ell}_{1,2} = \tilde{\ell}_{2,3} = \tilde{\ell}_{3,4}  = \tilde{\ell}_{1,5} = 1 = \tilde{\ell}_{1,3} = \tilde{\ell}_{1,4} = \tilde{\ell}_{2,4} = \tilde{\ell}_{2,5} = \tilde{\ell}_{3,5} =2 $ and $\tilde{\ell}_{4,5} = \frac{4\sqrt{6}}{3}$.
For $K_5-e$ we take typical parameter values $k_{i,j} = 1$, $\tilde{\ell}_{i,j}  = 2$.
}\label{fig:K4-3Dnetwork}
\end{figure}

\newpage

\subsection{Strategies for Solving Larger Spring Networks}\label{sec:big}
In this section we discuss strategies for solving larger systems. Since the number of rigid $2$-D graphs grows quite quickly, we only demonstrate our methods on a few examples.

An issue that arises when trying to use homotopy continuation to solve larger systems of polynomial networks is that this requires constructing an increasingly large start system, and the number of paths grows like the product of the degrees of the polynomials. The hc.jl library incorporates polyhedral methods to recognize structure in the system of equations, but these polyhedral start systems can also grow quickly with the problem size. Moreover, we noticed (see \Cref{tab:NetworkSummary2D,tab:NetworkSummary3D}) that the number of solutions is typically much smaller than the number of paths tracked, sometimes as low as 1-2\%. We adopt a strategy similar to the cheater homotopy strategy \cite{li1989cheater} to allow us to solve systems of equations whose start systems would be too large to attempt. Here is an outline of our method:
\begin{enumerate}
\item Find some initial solutions via one of the following methods
    \begin{enumerate}
        \item Using basic geometry and knowledge of the problem.
        \item Newton's method with (a large number of) random starting points.
        \item Starting a (hopelessly long) normal homotopy continuation solve but then stopping it early and saving the partial list of solutions.
    \end{enumerate}
\item Use the partial list of solutions as starting solutions for monodromy solving.
Note that for large systems, like $K_5$ standard solving can miss a lot of solutions. One run of that computation found 53\,228 non-singular solutions (4\,175 real), then applying monodromy to that solution set to find more solutions we found 53\,713 solutions before a heuristic stop was initiated after 5 loops were performed with no additional roots found, and apparently 4\,193 of them were real. 

\end{enumerate}

\section{Conclusion}

We evaluated a range of strategies for solving spring networks and compared them both qualitatively and quantitatively. Given a spring network we constructed two polynomial systems which represent it, the standard system and the inverse-length system. For both practical computations, and asymptotically as the number of nodes and edges in the spring system increases, the inverse-length system is more computationally efficient. 
Comparing the number of paths required for homotopy solving $K5-P4$ for the standard formulation versus the inverse-length formulation (\Cref{tab:NetworkSummary2D}) we see that it took 1.2M versus 30k paths, which is a 40x ratio improvement for inverse-lengths. We observe a similar ratio improvement in computational time. We similarly observe orders of magnitude improvement (in the 20x to 30x range) for the inverse-length formulation for the rest of the 2d graphs.

We gave bounds on the Bézout number for the standard formulation and the inverse-length formulations in \Cref{thm:Bézout}. 
Furthermore we demonstrated that there is an optimal choice of base points when constructing the polynomial system, namely choosing the base points to have highest degree, \Cref{thm:bound}. Even though the Bézout number for the inverse-length formulation is independent of base point, we see that the best choice of base points for the standard formulation is also best for the polyhedral bound for the inverse-length formulation. In fact, \Cref{tab:PolyhedralSummary2D} shows that there can be as much as a 2-fold difference in the number of paths needed to perform homotopy solving for the optimal choice of base points when compared to the least optimal choice.

We further evaluated the efficacy of a variety of methods for solving these spring systems. For small spring networks, using Newton's method or a combination of Newton's method and monodromy is optimal. For medium sized spring networks homotopy continuation is best. For large spring networks homotopy continuation is limited by numerical errors and memory constraints. Thus we have found, for large spring networks, it is best to perform a series of parameter homotopy continuations from a smaller spring network and perform monodromy after each homotopy to recover any missed solutions.

\section*{Funding acknowledgment}
This material is based on research sponsored by AFRL/RA under agreement number FA8651-25-2-0001. The U.S. Government is authorized to reproduce and distribute reprints for Governmental purposes notwithstanding any copyright notation thereon.

Distribution Statement A. Approved for public release: distribution is unlimited. Approved AFRL-2026-1489 31-08-2026.

JE and NW acknowledge support from the AFRL Scholars program.  LO acknowledges support from the SFFP program. LO and EC acknowledge support from the AFRL Innovative Research Fund.

\section*{Statement of author contributions}
All authors contributed to the writing of this work. 
LO and AR contributed to the problem conception and overall project management.  EC, LO, NW, JE contributed to the writing of versions of the code. EC contributed to the testing and reporting of the results. EC made all of the figures.

\clearpage
\appendix
\label{app1}

\newcommand{\arxiv}[1]{\href{http://arxiv.org/abs/#1}{{\tt arXiv:#1}}}

\bibliographystyle{elsarticle-harv}
\bibliography{main_bibfile}

\end{document}